\documentclass[a4 paper, 12pt, reqno]{amsart}
\usepackage{latexsym}
\usepackage[T1]{fontenc}
\usepackage[english]{babel}
\usepackage{amssymb,amsmath,amsthm,amsfonts}
\usepackage[latin1]{inputenc}

\usepackage{url}
\usepackage{a4wide}
\usepackage{enumerate}
\usepackage{changebar}
\usepackage{booktabs}
\usepackage{paralist}
\usepackage{comment}
\usepackage{mathrsfs}

\usepackage{graphicx}
\usepackage{color}
\usepackage{cite}

\theoremstyle{plain}
\newtheorem{theorem}{Theorem}[section]
\newtheorem{prop}[theorem]{Proposition}
\newtheorem{lem}[theorem]{Lemma}
\newtheorem{cor}[theorem]{Corollary}
\theoremstyle{definition}

\newtheorem{rem}[theorem]{Remark}

\def\R{{\mathbb R}}

\def\N{{\mathbb N}}

\def\D{{\mathcal{D}}}
\def\eps{\varepsilon}
\def\OU{\mathcal{L}}
\def\e{\mathrm{e}}

\def\d{\mathrm{d}}

\def\p{\mathrm p}
\def\div{\mathrm{div\,}}

\def\q{\mathrm{q}}

\numberwithin{equation}{section}

\title[On the Navier-Stokes equations with rotating effect ]
 {On the Navier-Stokes equations with rotating effect and prescribed outflow velocity}
\author[T. Hansel]{Tobias Hansel}

\address{%
International Research Training Group 1529\\ Technische Universit\"at Darmstadt\\ Schlossgartenstr. 7 \\ 64289 Darmstadt,
Germany}

\email{hansel@mathematik.tu-darmstadt.de}

\thanks{The author was supported by the DFG International Research Training Group 1529 \emph{Mathematical Fluid Dynamics} at TU Darmstadt.\\
To appear in J. Math. Fluid Mech.. Published online first. The
final publication is available at springerlink.com. }
\subjclass{Primary 35Q30; Secondary  76D03, 76D05}

\keywords{Navier-Stokes
flow, Oseen flow, rotating obstacle, non-autonomous PDE, evolution operators, Ornstein-Uhlenbeck
operator}

\begin{document}

\maketitle

\begin{abstract}
We consider the equations of Navier-Stokes modeling viscous fluid flow past a moving or rotating obstacle in $\R^d$ subject to a prescribed velocity condition at infinity. In contrast to previously known results, where the prescribed velocity vector is assumed to be parallel to the axis of rotation, in this paper we are interested in a general outflow velocity. In order to use $L^p$-techniques we introduce a new coordinate system, in which we obtain a non-autonomous partial differential equation with an unbounded drift term. We prove that the linearized problem in $\R^d$ is solved by an evolution
system on $L^p_{\sigma}(\mathbb R^d)$ for $1<p<\infty$. For this we use results about time-dependent Ornstein-Uhlenbeck operators. Finally, we prove, for $p\geq d$ and initial data $u_0\in L^p_{\sigma}(\R^d)$, the 
existence of a unique mild solution to the full Navier-Stokes system.
\end{abstract}


\section{Introduction}
The mathematical analysis of the Navier-Stokes flow past a rotating or moving obstacle has attracted quite some attention in recent 
years. It all started with the work of Borchers \cite{Borcher:1992}  in the framework of suitable weak solutions. Later Hishida \cite{Hishida:1999} constructed local mild solutions to the Navier-Stokes problem in the exterior of a rotating obstacle in the context of $L^2$ by using 
semigroup techniques (see also \cite{Hishida:2001}). This existence result was extended to the general $L^p$-theory by Geissert, Heck, 
Hieber \cite{Geissert/Heck/Hieber:2006} and Hishida, Shibata \cite{Hishida/Shibata:2009} showed that this solution is even a global one, 
provided the data are small enough. However, there are only a few partial results for the case when the fluid flow is subject to an 
additional outflow condition at infinity (hereby we mean a prescribed velocity of fluid at infinity). In fact, this situation was studied rather 
recently by Farwig \cite{Farwig:2005} and Shibata \cite{Shibata:2008} only for the special case when the outflow direction of the fluid is 
parallel to the axis of rotation of the obstacle. This assumption ensures -- after rewriting the problem on a fixed domain --  that the 
resulting equations are autonomous and thus can be treated e.g. by applying semigroup techniques. The purpose of this paper is to 
extend the existing results and to combine the rotating effect with a general outflow condition. For this purpose it is necessary to study the 
Navier-Stokes system perturbed by time-dependent and unbounded lower order terms, which is done here for the whole space case. 

To describe the situation more precisely, let $\mathcal O \subset \R^d$ be a compact obstacle with
smooth boundary and let $\Omega:=\R^d \setminus \mathcal O$ be the
exterior of the obstacle. We are interested in the case where the
obstacle undergoes a prescribed motion, particularly a rotation. 
So we let $M:[0,\infty) \to \R^{d\times d}$ be a
continuous matrix-valued function, such that $M(t)$ is
skew-symmetric for all $t>0$, i.e. $M(t)=-M^*(t)$,
and $M(t), M(s)$ commute\footnote{This condition can physically be interpreted by the fact that the axis of
rotation is fixed.} for all $t,s >0$. The exterior of the rotated obstacle at time $t>0$ is represented by $\Omega(t):=
U(t,0)\Omega$ where
\begin{equation}\label{eq:U(t,s)}
U(t,s) := \exp\Big( \int_s^t M(\tau) \d \tau\Big), \qquad t,s
\geq 0.
\end{equation}\normalsize
Since $M(t)$ is skew-symmetric for all $t>0$, the matrices $U(t,s)$ are orthogonal.  With a given velocity vector $v_\infty\in
\R^d\neq 0$, representing  the outflow velocity of the fluid, the Navier-Stokes equations on the time-dependent
domain $\Omega(t)$ with the usual no-slip boundary condition now
take the form
\begin{align}\label{eq:NS_2}
v_t-\Delta v +v\cdot \nabla v+\nabla \q&=0&\quad\quad\quad\mbox{in $\Omega(t)\times (0,\infty) $,}\notag\\
\div v&=0&\quad\quad\quad\mbox{in $\Omega(t)\times (0,\infty) $,}\notag\\
v(t,y)&=M(t)y&\quad\quad\quad\mbox{ on $\partial\Omega(t)\times (0,\infty) $,}\\
\lim_{|y|\to \infty} v(t,y)&=v_\infty\neq 0&\quad\quad\quad\mbox{ \mbox{for} $t\in(0,\infty) $,}\notag\\
v(0,y)&=u_0(y)&\quad\quad\quad\mbox{in $\Omega$}\vspace{0.3cm}\notag,
\end{align}
where $v$ and $\q$ are the unknown velocity field and the pressure of
the fluid, respectively. The disadvantage of this description is the
variability of the domain $\Omega(t)$, and the fact that the
equations do not fit into the $L^p$-setting, due the velocity
condition at infinity. By setting
\begin{equation}
x=U^*(t,0)y, \quad u(t,x)=U^*(t,0)
(v(t,y)-v_\infty), \quad \p(t,x)=\q(t,y),
\end{equation}
the above equations can be transformed back to the reference domain
$\Omega$ and the new velocity field $u$ vanishes at infinity.

\noindent We obtain the following system of equations:
\small
\begin{align}\label{eq:NS_3} \left.\begin{array}{l}
u_t- \Delta u - M(t)x \cdot \nabla u + M(t)u  \\
\quad+ U^*(t,0)v_{\infty}\cdot \nabla u +u\cdot \nabla u  +\nabla \p\end{array}\right\}&=0&\mbox{in   $\Omega \times (0,\infty)$,}\notag\\
\div u&=0&\quad\mbox{in $ \Omega \times (0,\infty) $,}\notag\\
u(t,x)&=M(t)x-U^*(t,0)v_\infty&\mbox{on $\partial\Omega \times (0,\infty)$},\\
\lim_{|x|\to \infty} u(t,x)&= 0&\mbox{\mbox{for} $t\in(0,\infty) $,}\notag\\
u(0,x)&=u_0(x)&\;\mbox{in $\Omega$}.\notag
\end{align}
\normalsize
The prize to pay for this transformation is that we obtain a non-autonomous partial differential equation with an unbounded drift term. 
Even if we assume that $M(t)\equiv M $ is independent of time, equation (\ref{eq:NS_3}) is still non-autonomous due to the time-
dependent first order term $U^*(t,0)v_{\infty}\cdot \nabla$. Only in the special situation where the velocity vector $v_{\infty}$ is parallel to 
the axis of rotation -- in this case $v_{\infty}$ is a fixed point under the transformation $U^*(t,0)$ -- the transformed equations remain 
autonomous. This shows that if one allows a general outflow condition, it is necessary to study a non-autonomous problem.

In the special case, where $M(t)x=\omega(t) \times x$ and
$\omega:[0,\infty) \rightarrow \R^3$ is the angular velocity of the
obstacle, Borchers \cite{Borcher:1992} constructed weak non-stationary solutions for the equations (\ref{eq:NS_3}). Later, Farwig \cite{Farwig:2005} studied the linearized
stationary problem with $\Omega=\R^d$ and he proved
$L^q$-estimates for the second derivative of the velocity field $u$
and for the first derivate of the pressure $\p$.  However, he only
considered the case, where $M(t)x = \omega \times x$ with
$\omega\in\R^3$ parallel to $v_{\infty}$. Recently, Shibata \cite{Shibata:2008}
proved, also for $M(t)x = \omega \times x$ with $\omega\in\R^3$
parallel to $v_{\infty}$, that the solution of the linearized
problem is governed by a strongly continuous semigroup on
$L^p_{\sigma}(\Omega)$, $1<p<\infty$, which is 
\emph{not analytic}.
His main result is actually the \emph{boundedness} of the semigroup (see also \cite{Hishida/Shibata:2009} for the case $v_{\infty}=0$).
By using Kato's iteration scheme (\cite{Kato:1984,Giga:1986})  this allows to prove the existence of a global solution to the full nonlinear problem for small initial
data. A time-dependent fundamental solution (Green's function) to problem (\ref{eq:NS_3}) was derived by Thomann, Guenther in \cite{Thomann/Guenther:2006} for the special case $M(t)x = \omega \times x$ with $\omega\in\R^3$
parallel to $v_{\infty}$.

Our approach to the non-autonomous equations (\ref{eq:NS_3}) is based on a linearization and on the family of modified time-dependent Stokes operators
$$
A(t)u :=\mathbb P \left( \Delta u + (M(t)x - U^*(t,0)v_{\infty}) \cdot \nabla u - M(t)u\right), \qquad
t>0,\vspace{0.2cm}
$$
where $\mathbb P$ denotes the Helmholtz-Leray projection from $L^p(\Omega)^d$ into $L^p_{\sigma}(\Omega)$, the space of all solenoidal vector fields in $L^p(\Omega)^d$ (see e.g.
\cite[Chapter III]{Galdi:1994}).
The main difficulty for treating operators of the above kind lies in
the fact that the coefficients of the drift term are unbounded and
thus the first order term cannot be consider as a ``small''
perturbation of the classical Stokes operator in unbounded domains. However, it
has been shown by Hieber, Sawada \cite{Hieber/Sawada:2005} for $\Omega=\R^d$ and by Geissert, Heck, Hieber \cite{Geissert/Heck/Hieber:2006} for exterior domains $\Omega$, that in the autonomous case, i.e. for fixed $t$, and for $v_{\infty}=0$, the operator $A(t)$ with an appropriate domain generates a strongly continuous
semigroup on $L^p_{\sigma}(\Omega)$, $1<p<\infty$, which is, however, not
analytic. The fact that the semigroup is
not analytic prevents us from employing standard generation results
for evolution systems of parabolic type mainly due to Tanabe \cite{Tanabe:1959, Tanabe:1960a, Tanabe:1960b} or Acquistapace, Terreni \cite{Acquistapace:1984,Acquistapace/Terreni:1986,Acquistapace/Terreni:1987} (see also \cite[Chapter
5]{Pazy:1983} or \cite[Chapter 6]{Tanabe:1997} for more information on this matter). Here lies one of the main difficulties. 
A first step in the study of the problem is to consider the whole space case rather than the physically more 
realistic situation of exterior domains. A solution to the whole space problem is not only interesting in its own right but also needed for using a cut-off technique to solve the exterior domain problem in a next step. Therefore, for the rest of this paper we study -- in a more general form -- the non-autonomous equations
\small
\begin{align} u_t-\Delta
u-\left(M(t)x + f(t)\right)\cdot\nabla u +M(t)u + u\cdot \nabla u + \nabla \p &=0&\quad\mbox{in $\R^d \times (0,\infty)$,}\notag\\
\div u&=0& \quad\mbox{in $\R^d \times (0,\infty)$,}\label{eq:NS_1}\\
u(0)&=u_0&\quad\mbox{in $\R^d$,}\vspace{0.3cm}\notag 
\end{align}\normalsize
where $M:[0,\infty) \rightarrow \mathcal\R^{d\times d}$, $f:[0,\infty) \rightarrow
\R^d$ are continuous functions and where we assume in addition\footnote{The physically reasonable condition that $M(t)$ is skew-symmetric for all $t>0$ is not needed for our main results and therefore not explicitly assumed for the rest of the paper unless otherwise stated.} that $M(t)M(s)=M(s)M(t)$ holds for all $t,s >0$. Here as usual,
 $u:\R^d \times (0,\infty)
\rightarrow \R^d$ and $\p: \R^d \times (0,\infty) \rightarrow \R$ denote the
unknown
velocity field and the pressure of the fluid respectively. By setting $f(t)=- U^*(t,0)v_{\infty}$ we are in the special situation of equation (\ref{eq:NS_3}).

This paper is organized as follows. In Section 2 we review and prove
results on time-dependent Ornstein-Uhlenbeck operators, studied recently by Da Prato, Lunardi
\cite{DaPrato/Lunardi:2007} and Geissert, Lunardi
\cite{Geissert/Lunardi:2007}. By using these results in Section 3 we
prove that the solution to the linearized problem is given by a strongly continuous evolution system on $L^p_{\sigma}(\R^d)$, $1<p<\infty$, and we derive
an explicit formula for the evolution operators, similar to the
representation formula known in the case of time-dependent
Ornstein-Uhlenbeck operators. Moreover, we prove $L^p$-$L^q$ as well
as gradient estimates for the evolution system. In Section 4 we return to the full Navier-Stokes problem (\ref{eq:NS_1}) and prove the
existence of a mild solution by adjusting Kato's iteration scheme to our situation.
\section{Time-dependent Ornstein-Uhlenbeck Operators}
In this section we assume that $M:\R \rightarrow \R^{d\times d}$ and $f:\R\rightarrow \R^d$
are continuous functions. Moreover, we define $\tilde{M}(t):= M(-t)$ for $t\in\R$ and denote by $U(t,s)$ and
$\tilde U(t,s)$ the solutions of the problems
\begin{equation}\label{eq:CP_U(t,s)}
\left\{
\begin{array}{rclll}
\frac{\partial}{\partial t}U(t,s) &=& M(t)U(t,s), &t,s\in\R,\\[0.1cm]
U(s,s) & = & I,
\end{array}\right.
\end{equation}
and
\begin{equation}
\left\{
\begin{array}{rclll}
\frac{\partial}{\partial t}\tilde{U}(t,s) &=& \tilde{M}(t)\tilde{U}(t,s), &t,s\in\R,\\[0.1cm]
\tilde{U}(s,s) & = & I,
\end{array}\right.
\end{equation}
respectively. 

Now we consider time-dependent Ornstein-Uhlenbeck
operators $\OU (t)$, formally defined on smooth functions $\varphi:\R^d\to\R$
by\vspace{0.05cm}
\begin{equation}
(\OU(t)\varphi)(x) = \Delta \varphi(x) + (M(t)x + f(t))\cdot
\nabla\varphi(x) , \quad t\in\R, \quad x\in \R^d,\vspace{0.05cm}
\end{equation}
and the associated non-autonomous forward Cauchy problem
\begin{equation}\label{eq:FCP}
\left\{
\begin{array}{rclll}
u_t(t,x)& = & \OU(t)u(t,x), & s< t,\; x \in \R^d, \\[0.1cm]
u(s,x) & = & \varphi(x), &  x\in\R^d,
\end{array}\right.
\end{equation}
where $s\in \R$ is fixed. A straightforward change of variables
allows to transform problem (\ref{eq:FCP}) into an equivalent
backward problem. More precisely, the function $(t,x)\mapsto u(t,x)$
is a classical solution to problem (\ref{eq:FCP}) if and only if the function 
$(t,x)\mapsto v(t,x):=u(-t,x)$ is a classical solution to the
backward problem\vspace{0.05cm}
\begin{equation}\label{eq:BCP}
\left\{
\begin{array}{rclll}
v_t(t,x) + \tilde{\OU}(t)v(t,x) &=& 0, & t< -s,\; x \in \R^d, \\[0.1cm]
v(-s,x) & = & \varphi(x), & x\in\R^d,
\end{array}\right.\vspace{0.05cm}
\end{equation}
where $\tilde{\OU}(t) := \OU(-t)$. Such a backward problem was
considered by Da Prato, Lunardi \cite{DaPrato/Lunardi:2007} and
Geissert, Lunardi \cite{Geissert/Lunardi:2007}, since their main motivation came from stochastics.
In our case, with
the application to problem (\ref{eq:NS_1}) in mind, it is more
convenient to work with the forward problem. The following
proposition follows, via the transformation mentioned above,
directly from the analogous result for the backward equation
(\ref{eq:BCP}) proved in \cite[Proposition
2.1]{DaPrato/Lunardi:2007}.
\begin{prop}\label{prop:OU_solution_formula}
Let $\varphi \in C_c^\infty(\R^d)$ and fix $s\in \R$. Then problem
(\ref{eq:FCP}) has a unique bounded classical solution $u\in
C^{1,2}([s,\infty)\times \R^d)$, given by the formula
\begin{equation}\label{eq:OUsol}
u(t,x)= \frac{1}{(4\pi)^{\frac d 2}(\det Q_{t,s})^{\frac 1 2}} \int_{\R^d}
\varphi(\tilde{U}(-s,-t)x+g(t,s)-y) \e^{-\frac 1 4 \langle
Q_{t,s}^{-1}y,y \rangle} \d y,
\end{equation}
where $g(t,s)$ and $Q_{t,s}$ are defined by
\begin{equation}\label{eq:DefQandg_1}
g(t,s)=\int_{-t}^{-s} \tilde{U}(-s,r)f(-r) \d r\quad \mbox{and}
\quad Q_{t,s}=\int_{-t}^{-s} \tilde{U}(-s,r)\tilde{U}^*(-s,r) \d r
\end{equation}
respectively.
\end{prop}
Note that the right hand side of (\ref{eq:OUsol}) is well defined
for each $L^p(\R^d)$-function $\varphi$. Thus, in the following this explicit formula serves as a starting point to define an \emph{evolution system} on $L^p(\R^d)$, $1< p < \infty$, associated with problem (\ref{eq:FCP}). Before, we have to give equation (\ref{eq:FCP}) a meaning in the $L^p$-setting, i.e., we have to define the $L^p$-realizations of the formally defined operators $\OU(t)$. For this purpose we set 
\begin{equation}\label{eq:OUdef}
\begin{array}{rcl}
\D(L(t))&:=& \{\varphi \in W^{2,p}(\R^d) :M(t)x\cdot \nabla\varphi(x) \in L^p(\R^d) \},\\
L(t)\varphi & := & \OU(t) \varphi.\vspace{0.1cm}
\end{array}
\end{equation}Here the domain of $L(t)$ is depending on $t$, but $C_c^{\infty}(\R^d)$
is a subset of $\D(L(t))$ for every $t\in\R$. It has been shown by
Metafune \cite{Metafune:2001} and
Metafune, Pr\"uss, Rhandi, Schnaubelt \cite{Metafune/etal:2002} that
in the autonomous case, i.e. for fixed $t$, and for $f(t)=0$, the operator $L(t)$ with
domain $\D(L(t))$ generates a strongly continuous semigroup on
$L^p(\R^d)$. However, due to the fact that the coefficients of the
drift term are unbounded this
semigroup is not analytic on $L^p(\R^d)$ in general. Thus, the existence of an evolution system with nice regularity properties does not follow from the general theory of parabolic evolution equations. However, formula  (\ref{eq:OUsol}) allows to define a family of operators as follows. For $\varphi \in L^p(\R^d)$ we put $G(s,s)\varphi=\varphi$ and for $t>s$ we define the operator $G(t,s)$ by  \small\vspace{0.1cm}
\begin{equation}\label{eq:OUevolsyst}
G(t,s)\varphi(x) := \frac{1}{(4\pi)^{\frac d 2}(\det Q_{t,s})^{\frac 1 2}}
\int_{\R^d} \varphi(\tilde{U}(-s,-t)x+g(t,s)-y) \e^{-\frac 1 4\langle Q_{t,s}^{-1}y,y \rangle} \d y,\vspace{0.1cm}
\end{equation}\normalsize
where $g(t,s)$ and $Q_{t,s}$ are defined as in (\ref{eq:DefQandg_1}).
\begin{lem}\label{lemma:OU_bounded_operators}
For $t\geq s$ fixed, the linear operator $G(t,s)$, defined in (\ref{eq:OUevolsyst}), is bounded on $L^p(\R^d)$, $1<p<\infty$. Moreover, $G(t,s)\varphi \in \D(L(t))$ holds for any $\varphi \in C_c^\infty(\R^d)$ and $t\geq s$.
\end{lem}
\begin{proof}
First let us note, that for
$\varphi \in L^p(\R^d)$ and $t> s$ we can write
$$
G(t,s)\varphi(x)= (\varphi \ast k_{t,s})(\tilde{U}(-s,-t)x +
g(t,s)), \quad  x\in\R^d,
$$ 
where the kernel $k_{t,s}$ is defined by
$$
k_{t,s}(x):= \frac{1}{(4\pi)^{\frac d 2}(\det Q_{t,s})^{\frac 1 2}} \e^{-\tfrac 1 4 \langle Q_{t,s}^{-1}x,x \rangle}, \quad
x\in\R^d.
$$
By a change of variable and Young's inequality we obtain
\begin{align}
\|G(t,s)\varphi\|_{L^p(\R^d)} & = \Big(\int_{\R^d} \big|\left(\varphi \ast
k_{t,s}\right)(\tilde{U}(-s,-t)x+g(t,s))\big|^p \d x \Big)^{\frac 1 p}
\nonumber \\
&= |\det \tilde{U}(-s,-t)|^{\frac 1 p}\Big(\int_{\R^d}
\left|\left(\varphi \ast k_{t,s}\right)(x)\right|^p \d x
\Big)^{\frac 1 p}
\nonumber \\
&\leq |\det \tilde{U}(-s,-t)|^{\frac 1 p} \|\varphi\|_{L^p(\R^d)} \|k_{t,s}\|_{L^1(\R^d)}\nonumber \\
& \leq
C \|\varphi\|_{L^p(\R^d)},\vspace{0.25cm}
\nonumber
\end{align}
for some constant $C>0$. This proves the first assertion.

\noindent To prove the second assertion it suffices to show $M(t)x \cdot \nabla (\varphi \ast k_{t,s})(x) \in L^p(\R^d)$, since $ \tilde{U}(-s,-t)$ is an invertible matrix. At first we note that\small
\begin{align*}
&\nabla\left(\varphi \ast k_{t,s}\right)(x)=\frac{1}{(4\pi)^{\frac d 2}(\det Q_{t,s})^{\frac 1 2}}
\int_{\R^d} \nabla \varphi(y)\e^{
-\tfrac 1 4  \big|Q_{t,s}^{-1/2}(x-y)\big|^2} \d y
\end{align*}\normalsize
holds. Now for a function $h\in L^{q}(\R^d)$ with $\frac 1 p + \frac{1}{q}=1$ we obtain\small
\begin{align*}
&\int_{\R^d}|\left(M(t)x \cdot \nabla(\varphi \ast k_{t,s})(x)\right)h(x)|\d x\\
&\quad\leq C \int_{\R^d} |\nabla\varphi(y)| \int_{\R^d} \big|M(t)x \, \exp
\big(-\tfrac 1 4 \big|Q_{t,s}^{-\frac 1 2}(x-y)\big|^2 \big) h(x) \big|\d x\, \d y\\
&\quad\leq C \int_{\R^d} |\nabla\varphi(y)| \int_{\R^d} \big|M(t)x\,  \exp
\big(-\tfrac 1 4 \big|Q_{t,s}^{-\frac 1 2}x\big|^2 -\tfrac 1 4 \big|Q_{t,s}^{-\frac 1 2}y\big|^2 + \tfrac 1 2 \langle x , y \rangle \big) h(x) \big|\d x\, \d y\\
&\quad\leq C \int_{\mathrm{supp}\;\varphi} \big|\nabla\varphi(y) \exp\big(-\tfrac 1 4 \Big|Q_{t,s}^{-1/2}y\big|^2 \big)\big|\d y \;\cdot\\
&\qquad \qquad\qquad \int_{\R^d} \big|M(t)x \, \exp
\big(-\tfrac 1 4 \big(\big|Q_{t,s}^{-\frac 1 2}x\big|^2  - 2 K |x|\big) \big) h(x) \big|\d x \vspace{0.25cm}
\end{align*}\normalsize
for constants $C,K>0$. Here we essentially used the fact that $\mathrm{supp}\;  \varphi$ is compact. Thus, we can conclude that\small
$$
\int_{\R^d}|\left(M(t)x \cdot \nabla(\varphi \ast k_{t,s})(x)\right)h(x)|\d x < \infty
$$\normalsize
holds for every $h\in L^{q}(\R^d)$ with $\frac 1 p + \frac{1}{q}=1$. This yields the assertion. 
\end{proof}
\noindent We are now in position to state the main result of this section.
\begin{prop}\label{prop:OUevolsyst}
Let $1< p < \infty$. The two parameter family of bounded linear operators $\{G(t,s): s\leq t\}$
defines an evolution system on $L^p(\R^d)$, i.e.,\vspace{0.1cm}
\begin{itemize}
\item[(i)] $G(s,s)=Id$ \hspace{0.2cm} and \hspace{0.2cm} $G(t,s) = G(t,r)G(r,s)$ \hspace{0.2cm} for  \hspace{0.2cm}  $-\infty < s \leq r \leq t <\infty $, \vspace{0.1cm}
\item[(ii)] for each $\varphi\in L^p(\R^d)$, \hspace{0.1cm} $(t,s) \mapsto
G(t,s)\varphi$ \hspace{0.1cm} is continuous on  \hspace{0.05cm} $-\infty < s \leq t < \infty$.
\end{itemize}\vspace{0.1cm}
Moreover, for any initial value $\varphi \in C_c^\infty(\R^d)$, the abstract non-autonomous Cauchy problem
\begin{equation}\label{eq:OU_ACP}\left\{
\begin{array}{rclll}
u'(t)& = & L(t)u(t), & s< t, \\[0.1cm]
u(s) & = & \varphi,
\end{array}\right.
\end{equation}
admits a classical solution $u$ given by $u(t) = G(t,s)\varphi$.
\end{prop}
\begin{proof}
In \cite[Proposition 2.4]{Geissert/Lunardi:2007} it was shown
that the law of evolution  
\begin{equation*}\label{eq:FunctEq}
G(t,s)G(s,r)\varphi = G(t,r)\varphi, \qquad \qquad r\leq s \leq t, 
\end{equation*}
holds for every $\varphi \in C_c^{\infty}(\R^d)$. Since $C_c^{\infty}(\R^d)$ is dense in $L^p(\R^d)$, property (i) follows.

In order to prove property (ii), we apply the change of the variable $y=Q_{t,s}^{1/2}z$, to see that
$$
G(t,s)\varphi(x)= \frac{(\det Q_{t,s})^{\frac 1 2}}{(4\pi)^{\frac d 2}(\det
Q_{t,s})^{\frac 1 2}} \int_{\R^d}
\varphi(\tilde{U}(-s,-t)x+g(t,s)-Q_{t,s}^{\frac 1 2}z) \e^{-\frac{|z|^2}{4}} \d z
$$
holds. For $t>s$ fixed, we pick two sequences $(t_n)_{n\in \N}$ and $(s_n)_{n\in \N}$ such that $t_n \geq s_n $ holds for every $n\in \N$ and $(t_n,s_n)\to (t,s)$ as $n \to \infty$. For every $\varphi \in C_c^{\infty}(\R^d)$ and every
$x\in\R^d$ we now obtain
$$
\varphi(\tilde{U}(-s_n,-t_n)x+g(t_n,s_n)-Q_{t_n,s_n}^{\frac 1 2}z)
\rightarrow \varphi(\tilde{U}(-s,-t)x+g(t,s)-Q_{t,s}^{\frac 1 2}z)
$$
as $n\rightarrow \infty$. Lebegue's theorem now yields
$G(t_n,s_n)\varphi \rightarrow G(t,s)\varphi$
as $n\rightarrow \infty$ for every $\varphi\in
C_c^{\infty}(\R^d)$. The density of $C_c^{\infty}(\R^d)$ in
$L^p(\R^d)$ yields (ii).

The last assertion follows directly from Proposition \ref{prop:OU_solution_formula} and Lemma \ref{lemma:OU_bounded_operators}. 
\end{proof}
In order to prove $L^p$-$L^q$ and gradient estimates in the
following section we need the following estimates for the matrices
$Q_{t,s}$.
\begin{lem}\label{lemma:Qts_Estimates}
For $0<T<\infty$ there exists a constant $C:=C(T)>0$ such that
\begin{itemize}
\item[(i)]
$\|Q_{t,s}^{-\frac 1 2}\| \leq C (t-s)^{-\frac 1 2}, \quad 0<s<t<T$,
\item[(ii)]$ (\det Q_{t,s})^{\frac 1 2} \geq
C(t-s)^{\frac d 2},  \quad 0<s<t<T$.
\end{itemize}
\end{lem}
Assertion (i) has been proved by Geissert and Lunardi \cite[Lemma
3.2]{Geissert/Lunardi:2007}. However, to make the paper as
self-contained as possible we provide a proof here.
\begin{proof}
Let $T>0$ and $x\in \R^d$. From (\ref{eq:DefQandg_1}) we obtain
\begin{equation*}
\langle Q_{t,s}x, x \rangle  = \int_{-t}^{-s}\langle \tilde U
(-s,r)\tilde U^* (-s,r)x, x\rangle \d r\ = \int_{-t}^{-s}\|\tilde
U^* (-s,r)x\|^2 \d r.
\end{equation*}
The continuity of the map $(-s,-t)\mapsto \tilde U (-s,-t)$
yields that there exists a $\delta > 0$ such that $\|\tilde U^*
(-s,-t)x-x\| \leq \frac 1 2 \|x\|$ for $t-s \leq \delta$. Thus
\begin{equation}\label{eq:Qts_1}
\langle Q_{t,s}x, x \rangle \geq \frac 1 4 (t-s) \|x\|^{2}
\end{equation}
holds for $0<t-s<\delta$. If $t-s \geq \delta$, we have
\begin{align}\label{eq:Qts_2}
\langle Q_{t,s}x, x \rangle & = \int_{-t}^{-s}\|\tilde U^* (-s,r)x\|^2 \d r \geq \int_{-s-\delta}^{-s}\|\tilde U^* (-s,r)x\|^2 \d r \notag\\
& \geq \frac 1 4 \delta \|x\|^2 \geq \frac{1}{4T} \delta (t-s)
\|x\|^2.
\end{align}
Since $Q_{t,s}$ is symmetric and positive definite, it follows from
(\ref{eq:Qts_1}) and (\ref{eq:Qts_2}) that
$$
\|Q_{t,s}^{-\frac 1 2}\| \leq C (t-s)^{-\frac 1 2}
$$
holds for all $0<s<t<T$ and a suitable constant $C>0$ depending on
$T$. To show assertion (ii) we first observe that $\det Q_{t,s}^{-1}
\leq C \|Q_{t,s}^{-1}\|^d$ holds for a suitable constant $C>0$. Thus
by applying (i) we obtain
$$
\det Q_{t,s} = \left(\det Q_{t,s}^{-1}\right)^{-1}\geq C_1
\left(\|Q_{t,s}\|^d \right)^{-1} \geq C_2 (t-s)^d,
$$
for constants $C_1,C_2>0$ and assertion (ii) directly follows.
\end{proof}
\noindent In the case that $M(t),M(s)$ commute for all $t,s\in\R$, we have $\tilde{U}(-s,-t) = U(t,s)$. This can easily
been seen, since in this case $U(t,s)$ has the explicit form
(\ref{eq:U(t,s)}). By a simple change of variables the representation
formula (\ref{eq:OUevolsyst}) can be rewritten to the following
form.
\begin{cor}\label{cor:OUevolsyst}
Let $M(t),M(s)$ commute for all $s,t \in \R$. Then for $\varphi \in
L^p(\R^d)$ and $t>s$ the evolution operator $G(t,s)$ associated with the non-autonomous Cauchy problem
(\ref{eq:OU_ACP}) is given by\small
\begin{equation}
G(t,s)\varphi(x):= \frac{1}{(4\pi)^{\frac d 2}(\det Q_{t,s})^{\frac 1 2}}
\int_{\R^d} \varphi(U(t,s)x+g(t,s)-y) \e^{-\frac 1 4 \langle
Q_{t,s}^{-1}y,y \rangle} \d y,
\end{equation}\normalsize
where $g(t,s)$ and $Q_{t,s}$ are defined by
\begin{equation}\label{eq:DefQandg}
g(t,s)=\int_{s}^{t} U(r,s)f(r) \d r\quad \mbox{and} \quad
Q_{t,s}=\int_{s}^{t} U(r,s)U^*(r,s) \d r,
\end{equation}
respectively.
\end{cor}
\section{The Linearized Problem: The Evolution System on $L^p_{\sigma}(\R^d)$}
From now on our standing assumption is that $M:[0,\infty)  \rightarrow
\R^{d\times d}$, $f:[0,\infty)\rightarrow \R^d$ are continuous and
$M(t), M(s)$ commute for all $t,s> 0$. We recall that in this
case the solution to problem (\ref{eq:CP_U(t,s)}) for $t,s\geq0$
is given by
\begin{equation}\label{eq:U(t,s)_2}
 U(t,s) = \exp\left( \int_s^t M(\tau) \d \tau\right).
\end{equation}
We define the family of linear operators $B(t)$, $t>0$,
in $L^p(\R^d)^d$, $1<p<\infty$, by
\begin{equation}\label{eq:operatorB}
\begin{array}{rcl}
\D(B(t))&:=&\D(L(t))^d,\\
B(t)u & := & D_{L(t)}u - M(t)u,
\end{array}
\end{equation}
where $u=(u_1,\ldots,u_d)\in L^p(\R^d)^d$. Here $D_{L(t)}$ is the
$d\times d$ diagonal matrix operator with entries $L(t)$, defined as
in (\ref{eq:OUdef}). 
For $u\in L^p(\R^d)^d$ we put $W(s,s)u=u$ and for $0\leq s < t$ we define
\begin{align}\label{eq:EvolSyst}
W(t,s)u(x) &= \frac{1}{(4\pi)^{\frac d 2}(\det Q_{t,s})^{\frac 1 2}}U(s,t)\cdot
\int_{\R^d} u(U(t,s)x+g(t,s)-y)\notag \\
& \qquad \times  \e^{-\frac 1 4 \langle Q_{t,s}^{-1}y,y
\rangle} \d y, \qquad \qquad\qquad x\in \R^d,
\end{align}\normalsize
where $g(t,s)$ and $Q_{t,s}$ are defined as in (\ref{eq:DefQandg}). Analogously to Lemma \ref{lemma:OU_bounded_operators} it follows that, for $0\leq s \leq t$, the operator $W(t,s)$ is well defined and bounded on $L^p(\R^d)^d$. Based on Proposition \ref{prop:OUevolsyst} and Corollary \ref{cor:OUevolsyst} we now obtain the following result.
\begin{prop}\label{prop:EvolSyst}
Let $1<p<\infty$. The two parameter family of bounded linear operators $\{W(t,s): 0\leq s\leq t\}$
defines an evolution system on $L^p(\R^d)^d$, i.e.,\vspace{0.1cm}
\begin{itemize}
\item[(i)] $W(s,s)=Id$ \hspace{0.2cm} and \hspace{0.2cm} $W(t,s) = W(t,r)W(r,s)$ \hspace{0.2cm} for  \hspace{0.2cm}  $0\leq s \leq r \leq t  < \infty$, \vspace{0.1cm}
\item[(ii)] for each $u\in L^p(\R^d)^d$, \hspace{0.1cm} $(t,s) \mapsto
W(t,s)u$ \hspace{0.1cm} is continuous on  \hspace{0.1cm} $0\leq  s \leq t < \infty$.
\end{itemize}\vspace{0.1cm}
Moreover, for any initial value $\varphi \in C_c^\infty(\R^d)^d$, the abstract non-autonomous Cauchy problem
\begin{equation}\label{eq:prop_evolution_W_nACP}\left\{
\begin{array}{rclll}
u'(t)& = & B(t)u(t), &0\leq s < t, \\[0.05cm]
u(s) & = & \varphi,
\end{array}\right.
\end{equation}
admits a classical solution $u$ given by $u(t) = W(t,s)\varphi$.
\end{prop}
\begin{proof}
For $u\in L^p(\R^d)^d$, $t> s$, and
$x\in\R^d$ we define the operator $\tilde G(t,s)$ by
$$
\tilde G(t,s)u(x) := \frac{1}{(4\pi)^{\frac d 2}(\det Q_{t,s})^{\frac 1 2}}
\int_{\R^d} u(U(t,s)x+g(t,s)-y) \e^{-\frac 1 4 \langle
Q_{t,s}^{-1}y,y \rangle} \d y.
$$
This is just the Ornstein-Uhlenbeck evolution system from Proposition \ref{prop:OUevolsyst}  applied in each
component of the function $u=(u_1,\ldots,u_d)$. Thus, $\{\tilde
G(t,s): 0 \leq s \leq t\}$ is
an evolution system on $L^p(\R^d)^d$ such that 
$$
\frac{\partial}{\partial t} \tilde
G(t,s)\varphi = D_{L(t)}\tilde
G(t,s)\varphi
$$
holds for every  $\varphi \in C_c^\infty(\R^d)^d$. 
Note that for $u\in L^p(\R^d)^d$ and $t\geq s$ we can
write\footnote{To be precise, $U(s,t)$ has to be interpreted here as
a multiplication operator.} $W(t,s)u=U(s,t)\tilde G(t,s)u$. By
applying the product rule we obtain
\begin{align}
\frac{\partial}{\partial t}W(t,s)u &= \frac{\partial}{\partial t}U(s,t)\tilde G(t,s)u \notag\\
& = U(s,t) D_{L(t)}\tilde G(t,s)u -
U(s,t)M(t)\tilde G(t,s)u \nonumber \\
& = B(t)W(t,s)u, \nonumber
\end{align}
for every $u\in C_c^{\infty}(\R^d)^d$. We have used that
$D_{L(t)}-M(t)$ commutes with the multiplication by $U(s,t)$, which
can be easily seen, as $U(t,s)$ is given by (\ref{eq:U(t,s)_2}). Thus for every $u\in C_c^\infty(\R^d)$ the solution to equation (\ref{eq:prop_evolution_W_nACP}) is indeed given by $W(t,s)u$.

The law of evolution follows from a similar calculation. For $0\leq s\leq r\leq t$ we have
\begin{align}
W(t,r)W(r,s)u & = U(r,t)\tilde G(t,r)\left(U(s,r)\tilde G(r,s)u \right) \nonumber \\
& = U(r,t)U(s,r)\tilde G(t,r)\tilde G(r,s)u \notag\\
& = W(t,s)u. \nonumber
\end{align}
Here we have used $U(r,t)U(s,r)=U(s,t)$, which also can 
be seen from (\ref{eq:U(t,s)_2}). 

The strong continuity of $(t,s) \mapsto W(t,s)$ follows directly
from the strong continuity of $(t,s) \mapsto U(s,t)$ and
$(t,s)\mapsto \tilde G(t,s)$. This completes the proof.
\end{proof}
By the Proposition \ref{prop:EvolSyst}
 $\{W(t,s):0\leq s\leq t\}$ is an evolution system on $L^p(\R^d)^d$. However, later in Section 4 we shall not work on
$L^p(\R^d)^d$ but, as usual in the theory of the Navier-Stokes equations, on $L^p_{\sigma}(\R^d)$, the space of all solenoidal vector fields in $L^p(\R^d)^d$. Therefore we also consider the operators
$A(t)$, $t>0$, in $L^p_{\sigma}(\R^d)$ defined by
\begin{equation}\label{eq:operatorB}
\begin{array}{rcl}
A(t) & := & B(t)|_{L^p_{\sigma}(\R^d)},\\
\D(A(t))&:=&\D(B(t))\cap L^p_\sigma(\R^d),
\end{array}
\end{equation}
i.e., $A(t)$ is the restriction of $B(t)$ to $L^p_{\sigma}(\R^d)$. 
To ensure that this definition really makes sense we have to show
that the operators $B(t)$, $t>0$, leave $L^p_\sigma (\R^d)$ invariant. An easy calculation shows
that
\begin{equation}
\div\{M(t)x \cdot \nabla u + f(t)\cdot \nabla u - M(t)u\} = 0
\end{equation}
holds for all $u\in C_{c,\sigma}^\infty(\R^d)$. Thus, $A(t)$, $t>0$, is indeed a linear operator acting on $L^p_{\sigma}(\R^d)$. Similarly, we can show that  
\begin{equation}\label{eq:invariant}
\div \left(U(s,t)\cdot u(U(t,s)x+g(t,s)\right) = \left( \div 
u\right) (U(t,s)x+g(t,s)) = 0
\end{equation}
holds for all $u\in C_{c,\sigma}^\infty(\R^d)$. It now easily follows from (\ref{eq:invariant}) that also the evolution system $\{W(t,s):0\leq s \leq t\}$ leaves $L^p_{\sigma}(\R^d)$ invariant. Thus we can define a family of operators on $L^p_{\sigma}(\R^d)$ by setting 
$$
V(t,s) =W(t,s)|_{L^p_{\sigma}(\R^d)},\qquad \qquad 0\leq s\leq t,
$$
i.e., $V(t,s)$ is just the restriction of $W(t,s)$ to $L^p_\sigma (\R^d)$. The next result now follows directly from Proposition \ref{prop:EvolSyst}. 
\begin{prop}
Let $1<p<\infty$. The two parameter family of bounded linear operators $\{V(t,s): 0\leq s\leq t\}$
defines an evolution system on $L^p_\sigma(\R^d)$.
Moreover, for any initial value $\varphi \in C_{c,\sigma}^\infty(\R^d)$, the abstract non-autonomous Cauchy problem
\begin{equation}\label{eq:prop_evolution_W_nACP}\left\{
\begin{array}{rclll}
u'(t)& = & A(t)u(t), & 0\leq s < t, \\[0.1cm]
u(s) & = & \varphi,
\end{array}\right.
\end{equation}
admits a classical solution $u$ given by $u(t) = V(t,s)\varphi$.
\end{prop}
This shows that the Stokes problem corresponding to equation (\ref{eq:NS_1}) is solved by the evolution system $\{V(t,s): 0\leq s \leq t\}$ on $L^p_\sigma(\R^d)$.  Next we prove $L^p$-$L^q$ and gradient estimates for this evolution system. Since the evolution system  is not of parabolic type in the sense of Tanabe or  Acquistapace, Terreni, gradient estimates do not follow from the general theory. However, the explicit formula for $V(t,s)$ allows us
to obtain the following result.
\begin{prop}
Let $1<p<\infty$ and $p\leq q \leq \infty$.
\begin{itemize}
\item[(a)] For $T>0$ there exists a constant $C>0$ such that for $u\in L^p_{\sigma}(\R^d)$
\begin{align}
\|V(t,s)u\|_{L_{\sigma}^q(\R^d)} \leq C
(t-s)^{-\frac{d}{2}\left(\frac{1}{p}-\frac{1}{q}\right)}\|u\|_{L_{\sigma}^p(\R^d)},&
\quad \mbox{for} \quad 0\leq s<t\leq T,\label{eq:LpLqsmoothing}\\
  \|\nabla V(t,s)u\|_{L^q(\R^d)} \leq C
(t-s)^{-\frac{d}{2}\left(\frac{1}{p}-\frac{1}{q}\right)-\frac{1}{2}}\|u\|_{L_{\sigma}^p(\R^d)},
&\quad \mbox{for} \quad 0\leq s<t\leq T. \label{eq:GradientEstimate}
\end{align}
\item[(b)] Assume in addition that $M(t)$ is skew-symmetric for all $t>0$. Then there exists a constant
$C>0$ such that  for $u\in L^p_{\sigma}(\R^d)$
\begin{align}
\|V(t,s)u\|_{L_{\sigma}^q(\R^d)} \leq C
(t-s)^{-\frac{d}{2}\left(\frac{1}{p}-\frac{1}{q}\right)}\|u\|_{L_{\sigma}^p(\R^d)},&
\quad \mbox{for} \quad 0\leq s<t,\label{eq:LpLqsmoothing2}\\
\|\nabla V(t,s)u\|_{L^q(\R^d)} \leq C
(t-s)^{-\frac{d}{2}\left(\frac{1}{p}-\frac{1}{q}\right)-\frac{1}{2}}\|u\|_{L_{\sigma}^p(\R^d)},
&\quad \mbox{for}\quad 0\leq s<t.\label{eq:GradientEstimate2}
\end{align}
\end{itemize}
\end{prop}
\begin{proof}
We start by showing (\ref{eq:LpLqsmoothing}). Let $T>0$. By a change
of variables and by Young's inequality we obtain\small
\begin{align*}
& \|V(t,s)u\|_{L_{\sigma}^q(\R^d)} \leq \frac{\|U(s,t)\|}{(4\pi)^{\frac d 2} (\det
Q_{t,s})^{\frac 1 2 }} \, |\det
U(t,s)|^{\frac 1 q}\Big(\int_{\R^d}\big|\e^{-\frac 1 4 \langle
Q_{t,s}^{-1}y,y \rangle}\big|^r \d y\Big)^{\frac 1 r}
\|u\|_{L_{\sigma}^p(\R^d)},
\end{align*}\normalsize
where $1<r<\infty$ with $\frac 1 p + \frac 1 r = 1 + \frac 1 q$.
Further, by the change of variable $y= Q_{t,s}^{1/2} z$ we obtain\small
\begin{align}
\Big(\int_{\R^d}\big|\e^{-\frac 1 4 \langle Q_{t,s}^{-1}y,y
\rangle}\big|^r \d y\Big)^{\frac 1 r} &= \Big(\int_{\R^d}
\e^{-\frac{r |z|^2}{4}} (\det Q_{t,s})^{1/2} \d
z\Big)^{\frac 1 r} \leq C (\det Q_{t,s})^{\frac{1}{2r}}, \nonumber
\end{align}\normalsize
for some constant $C>0$. Now Lemma \ref{lemma:Qts_Estimates} (ii)
yields the assertion.

To prove the gradient estimate (\ref{eq:GradientEstimate}), we first
observe that\small
\begin{align}\label{eq:gradient}
&\nabla V(t,s)u(x)= \frac{U(s,t)}{(4\pi)^{\frac d 2} (\det Q_{t,s})^{\frac 1
2}} \int_{\R^d} u(U(t,s)x+g(t,s)-y) \nabla\e^{-\frac 1 4 \langle Q_{t,s}^{-1}y,y \rangle} U(t,s) \d
y\nonumber 
\end{align}\normalsize
holds. Similarly as above we now obtain the desired estimate\small
\begin{align*}
& \|\nabla V(t,s) u\|_{L^q(\R^d)} \\
& \; \leq \frac{\|U(s,t)\| \|U^*(t,s)\| 
}{(4\pi)^{\frac d 2} (\det Q_{t,s})^{\frac{1}{2}}} |\det U(t,s)|^{\frac 1 q}\Big(\int_{\R^d}
\big| \nabla \e^{- \frac 1 4 \langle
Q_{t,s}^{-1}y,y\rangle}\big|^r \d
y\Big)^{\frac 1 r} \|u\|_{L^p_\sigma(\R^d)}\nonumber \\
& \; \leq \frac{\|U(s,t)\| \|U^*(t,s)\| 
}{(4\pi)^{\frac d 2} (\det Q_{t,s})^{\frac{1}{2}}} |\det U(t,s)|^{\frac 1 q}\Big(\int_{\R^d}
\big|\Big(-\tfrac 1 2Q^{-1}_{t,s}y\Big)\, \e^{- \frac 1 4 \langle
Q_{t,s}^{-1}y,y\rangle}\big|^r \d
y\Big)^{\frac 1 r} \|u\|_{L^p_\sigma(\R^d)}\nonumber \\
& \; \leq \frac{\|U(s,t)\| \|U^*(t,s)\| 
}{(4\pi)^{\frac d 2} (\det Q_{t,s})^{\frac{1}{2}}} |\det U(t,s)|^{\frac 1 q} \|Q_{t,s}^{-\frac 1 2}\|  \Big(\int_{\R^d} |z|^r
\e^{-\frac{r|z|^2}{4}} (\det Q_{t,s})^{\frac 1 2}\d z\Big)^{\frac 1 r} \|u\|_{L^p_\sigma(\R^d)}\nonumber \\
& \;\leq C
(t-s)^{-\frac{d}{2}\left(\frac{1}{p}-\frac{1}{q}\right)-\frac{1}{2}}\|u\|_{L_{\sigma}^p(\R^d)},
\nonumber
\end{align*}\normalsize
for some constant $C>0$. Here we used Lemma
\ref{lemma:Qts_Estimates} (i) and (ii).

In order to prove (\ref{eq:LpLqsmoothing2}) and
(\ref{eq:GradientEstimate2}) we first note, that the fact that
$M(t)$ is skew-symmetric for all $t>0$ implies that the
evolution operator $U(t,s)$ is orthogonal for all $t,s>0$. Thus $\|U(t,s)\|=1$ and $|\det U(t,s)| =1$ holds for all $t,s>0$.
Moreover, we have $Q_{t,s}=(t-s)I$ for all $0<s<t$ and therefore it is trivial
that the estimates in Lemma \ref{lemma:Qts_Estimates} hold for all
$0<s<t$. The estimates (\ref{eq:LpLqsmoothing2}) and
(\ref{eq:GradientEstimate2}) now follow from the calculations above.
\end{proof}
\begin{prop}\label{prop:estimatesEvolutionSys}
For $1<p<q<\infty$ and $u\in L^p_{\sigma}(\R^d)$
\begin{equation}\label{eq:behavior_ts_1}
(t-s)^{\frac d 2 \left(\frac 1 p - \frac 1 q\right)} \|V(t,s)u\|_{L^q_{\sigma}(\R^d)}
\rightarrow 0 \quad \mbox{as} \quad t\rightarrow s \quad \mbox{and}
\end{equation}
\begin{equation}\label{eq:behavior_ts_2}
(t-s)^{\frac 1 2 } \|\nabla V(t,s)u\|_{L^p(\R^d)} \rightarrow 0 \quad
\mbox{as}\quad t\rightarrow s.
\end{equation}
\end{prop}
\begin{proof}
Let $t-s \leq 1$ and $u_n \in C_{c,\sigma}^{\infty}(\R^d) \subset L_{\sigma}^p(\R^d)
\cap L_{\sigma}^q(\R^d)$, $n\in \N$, such that $u_n \rightarrow u$ in
$L^p(\R^d)$ as $n\rightarrow \infty$. The triangle inequality
together with the $L^p$-$L^q$ estimates (\ref{eq:LpLqsmoothing})
imply that there exist constants $C_1, C_2>0$ such that 
\begin{align*}
&(t-s)^{\frac d 2 \left(\frac 1 p - \frac 1 q \right)} \|V(t,s)u\|_{L_{\sigma}^q(\R^d)}\\
& \quad \leq (t-s)^{\frac d 2 \left(\frac 1 p - \frac 1 q \right)}
\|V(t,s)u-V(t,s)u_n\|_{L_{\sigma}^q(\R^d)} + (t-s)^{\frac d 2 \left(\frac 1 p - \frac 1
q
\right)} \|V(t,s)u_n\|_{L_{\sigma}^q(\R^d)}\\
&\quad \leq  C_1 \|u-u_n\|_{L_{\sigma}^p(\R^d)} + C_2 (t-s)^{\frac d 2 \left(\frac 1 p -
\frac 1 q \right)} \|u_n\|_{L_{\sigma}^q(\R^d)} \rightarrow 0,
\end{align*}\normalsize
by letting first $t\rightarrow s$ and then $n\rightarrow \infty$.

Similarly, by using (\ref{eq:LpLqsmoothing}) and
(\ref{eq:GradientEstimate}) we obtain
\begin{align}
&(t-s)^{\frac 1 2} \|\nabla V(t,s)u\|_{L^p(\R^d)} \nonumber \\
&\quad \leq (t-s)^{\frac 1 2}
\|\nabla V(t,s)u-\nabla V(t,s)u_n\|_{L^p(\R^d)} + (t-s)^{\frac 1 2} \|\nabla V(t,s)u_n\|_{L^p(\R^d)} \nonumber \\
&\quad \leq  C_1 \|u-u_n\|_{L_{\sigma}^p(\R^d)} + (t-s)^{\frac 1 2} \|\nabla V(t,s)u_n\|_{L^p(\R^d)}.\label{eq:estimates_EvSys}
\end{align}\normalsize
Since $u_n \in C_{c,\sigma}^{\infty}(\R^d)$, we observe that
\begin{align}
&\nabla V(t,s)u_n(x)\nonumber\\
&\quad= \frac{U(s,t)}{(4\pi)^{\frac d 2} (\det Q_{t,s})^{\frac 1
2}} \int_{\R^d} \nabla u_n(U(t,s)x+g(t,s)-y) \e^{-\frac 1 4 \langle Q_{t,s}^{-1}y,y \rangle} U(t,s) \d
y\nonumber 
\end{align}\normalsize
holds. Thus, as in the proof of estimate (\ref{eq:LpLqsmoothing}) we now obtain
\begin{equation*}
\|\nabla V(t,s)u_n\|_{L^p(\R^d)} \leq C_2 \|\nabla u_n\|_{L^p(\R^d)}
\end{equation*}
for some constant $C_2$. Now the assertion follows from (\ref{eq:estimates_EvSys}) by letting $t\rightarrow s$ and $n\rightarrow \infty$.
\end{proof}
\section{The Navier-Stokes Flow}
By applying the Helmholtz-Leray projection $\mathbb{P}$ to (\ref{eq:NS_1}) the pressure $\p$
can be eliminated and we may rewrite the equations as a
non-autonomous Cauchy problem\vspace{0.2cm}
\begin{equation}\label{eq:NS_abstract}
\left\{
\begin{array}{rclll}
u'(t)- A(t)u(t) + \mathbb{P}( (u(t) \cdot \nabla) u(t) )&=&0, & \text{for } t>0,  \\[0.1cm]
u(0)&=&u_0,&
\end{array}\right.\vspace{0.2cm}
\end{equation}
with initial value $u_0 \in L^p_{\sigma}(\R^d)$. By the Duhamel
principle this problem is reduced to the integral equation
\begin{equation}\label{eq:Duhamel}
u(t) = V(t,0)u_0 - \int_0^t V(t,s) \mathbb P((u(s) \cdot \nabla)
u(s)) \d s, \quad t>0,
\end{equation}
in $L^p_{\sigma}(\R^d)$. In the following, given $0<T_0\leq \infty$,
we call $u\in C([0,T_0);L^p_{\sigma}(\R^d))$ a
\textit{mild solution} of (\ref{eq:NS_abstract}) if $u$ satisfies
the integral equation (\ref{eq:Duhamel}) on  $[0,T_0)$. By adjusting
Kato's iteration scheme (\cite{Kato:1984,Giga:1986}) to our
situation we now prove the existence of a unique (local) mild solution.\vspace{0.05cm}
\begin{prop}\label{prop:Kato}
Let $2\leq d\leq p \leq q < \infty$ such that $d\neq q$ and $u_0\in
L^p_{\sigma}(\R^d)$. Then there exists $T_0>0$ and a unique mild
solution $u\in C([0,T_0);L^p_{\sigma}(\R^d))$ of
(\ref{eq:NS_abstract}), which has the properties
\begin{equation}\label{eq:mild_solution_prop1}
t^{\frac{d}{2}\left(\frac{1}{p}-\frac{1}{q}\right)}u(t) \in
C([0,T_0);L^q_{\sigma}(\R^d)) ,
\end{equation}
\begin{equation}\label{eq:mild_solution_prop2}
t^{\frac{d}{2}\left(\frac{1}{p}-\frac{1}{q}\right)+\frac{1}{2}}\nabla
u(t) \in C([0,T_0);L^q(\R^d)^{d\times d});
\end{equation}
if $p<q$, then
\begin{equation}\label{eq:mild_solution_prop3}
t^{\frac{d}{2}\left(\frac{1}{p}-\frac{1}{q}\right)}\|u(t)\|_{L^q(\R^d)} +
t^{\frac{1}{2}}\|\nabla u(t)\|_{L^p(\R^d)} \rightarrow 0 \quad \mbox{as\;} t
\rightarrow 0 .
\end{equation}
\end{prop}
\begin{rem}
In the case $p>d$, property (\ref{eq:mild_solution_prop3}) is not
necessary to guarantee the uniqueness of the mild solution $u$.
\end{rem}
\begin{proof}[Proof of Proposition \ref{prop:Kato}]
Let $q>p\geq d$ or $q\geq p >d$ and take $u_0 \in
L^p_{\sigma}(\R^d)$ and $T>0$. We set $u_1(t)=V(t,0)u_0$ and for
$j\geq 1$ and $t>0$ we define a recursion by
\begin{equation}\label{eq:Iteration}
u_{j+1}(t)=V(t,0)u_0 - \int_0^t V(t,s)\mathbb P(( u_j(s) \cdot
\nabla) u_j(s))\d s.
\end{equation}
Our aim is to show that for some $0 <T_0 \leq T$, this sequence
converges in $C([0,T_0);L^p_{\sigma}(\R^d))$ to a solution $u$ of
(\ref{eq:Duhamel}).

We set $\gamma = \frac{d}{2}\left(\frac 1 p - \frac 1 q\right)$ and
for $j\geq 1$ we define constants
$$
K_j := K_j(T_0):=\sup_{0<t\leq T_0} t^{\gamma} \|u_j\|_{L^q(\R^d)}, 
$$
$$
K_j':=K_j'(T_0) := \sup_{0<t\leq T_0} t^{\frac{1}{2}} \|\nabla
u_j\|_{L^p(\R^d)}
$$
and
$$
L_j:= L_j(T_0):=\sup_{0<t\leq T_0}
t^{\gamma}\|u_{j+1}(t)-u_j(t)\|_{L^q(\R^d)}, 
$$
$$
L_j':=
L_j'(T_0):=\sup_{0<t\leq T_0} t^{\frac{1}{2}}\|\nabla u_{j+1}(t)-
\nabla u_j(t)\|_{L^p(\R^d)}.
$$
Moreover, we set $R_j:=R_j(T_0):=\max\{K_j,K_j'\}$ . Note that the
$L^p$-$L^q$ estimates (\ref{eq:LpLqsmoothing}) and the gradient
estimates (\ref{eq:GradientEstimate}) yield
$R_1\leq C\|u_0\|_{L^p(\R^d)}$ for some constant $C>0$.

From (\ref{eq:Iteration}), the $L^r$-$L^q$ estimates
(\ref{eq:LpLqsmoothing}) and the boundedness of $\mathbb P$ from
$L^r(\R^d)^d$ into $L^r_{\sigma}(\R^d)$ it follows that
\begin{align}
&\|u_{j+1}(t)\|_{L^q(\R^d)} \nonumber\\
&\qquad \leq \|V(t,0)u_0\|_{L^q(\R^d)} + \int_0^t \| V(t,s)\mathbb P(( u_j(s) \cdot \nabla) u_j(s)) \|_{L^q(\R^d)} \d s\nonumber\\
& \qquad \leq t^{-\gamma}K_1 + C \int_0^t (t-s)^{-\frac{d}{2}\left(\frac 1
r - \frac{1}{q}\right)}\|( u_j(s) \cdot \nabla) u_j(s)\|_{L^r(\R^d)} \d
s,\label{eq:proofKato1}
\end{align}
holds, where $\frac 1 r=\frac{1}{p}+\frac 1 q$. Similarly, with the
gradient estimate (\ref{eq:GradientEstimate}) we obtain
\begin{align}
&\|\nabla u_{j+1}(t)\|_{L^p(\R^d)} \nonumber\\
&\qquad \leq \|\nabla V(t,0)u_0\|_{L^p(\R^d)} + \int_0^t \| \nabla V(t,s)\mathbb P(( u_j(s) \cdot \nabla) u_j(s)) \|_{L^p(\R^d)} \d s\nonumber\\
& \qquad \leq t^{-\frac{1}{2}}K_1' + C \int_0^t
(t-s)^{-\frac{d}{2}\left(\frac 1 r - \frac{1}{p}\right)-\frac 1
2}\|( u_j(s) \cdot \nabla) u_j(s)\|_{L^r(\R^d)} \d s.\label{eq:proofKato2}
\end{align}
In order to estimate the terms on the right hand side of the
inequalities (\ref{eq:proofKato1}) and (\ref{eq:proofKato2}), we
apply H\"older's inequality to conclude
\begin{equation}\label{eq:Hoelder}
\|( u_j(s) \cdot \nabla) u_j(s)\|_{L^r(\R^d)}  \leq \|u_j(s)\|_{L^q(\R^d)} \|\nabla
u_j(s)\|_{L^p(\R^d)} \leq K_jK_j's^{-\gamma - \frac 1 2}.
\end{equation}
This implies
\begin{equation}\label{eq:proofKato3}
\|u_{j+1}(t)\|_{L^q(\R^d)} \leq  t^{- \gamma}K_1 + C K_jK_j'  \int_0^t
(t-s)^{- \frac{d}{2p}} s^{-\gamma - \frac 1 2} \d s,
\end{equation}
and
\begin{equation}\label{eq:proofKato4}
\|\nabla u_{j+1}(t)\|_{L^p(\R^d)} \leq  t^{- \frac{1}{2}}K_1' + C K_jK_j'
\int_0^t (t-s)^{- \frac{d}{2q}-\frac1 2}  s^{-\gamma - \frac 1 2} \d
s,
\end{equation}
respectively. By multiplying inequality (\ref{eq:proofKato3}) with
$t^{\gamma}$ and inequality (\ref{eq:proofKato4}) with $t^{\frac 1
2}$ and then by taking $\sup_{0<t\leq T_0}$ we obtain
\begin{equation}\label{eq:proofKato5}
K_{j+1}\leq K_1 + C_1 K_jK_j' \quad \mbox{and} \quad K'_{j+1}\leq
K_1' + C_2 K_jK_j'
\end{equation}
for some positive constants $C_1, C_2$ independent of
$j$, but depending on $T$. Here we have used the
estimate
\begin{align*}
\int_0^t (t-s)^{- \alpha}s^{-\beta}\d s & = \int_{t/2}^t (t-s)^{- \alpha}s^{-\beta}\d s + \int_0^{t/2} (t-s)^{- \alpha}s^{-\beta}\d s \\
& \leq \left(\frac{t}{2}\right)^{-\beta} \int_{t/2}^t (t-s)^{- \alpha} \d s+\left(\frac{t}{2}\right)^{-\alpha} \int_0^{t/2} s^{-\beta}\d s \\
& \leq
\left(\frac{t}{2}\right)^{1-\beta-\alpha}\left(\frac{1}{1-\alpha}+\frac{1}{1-\beta}\right),
\end{align*}
for exponents $0<\alpha,\beta<1$.

From (\ref{eq:proofKato5}) it now follows that $R_{j+1} \leq R_1 +
\delta R_j^2$ holds, for some positive constant $\delta\geq 1$.  If
we assume $R_1 \leq \frac{1}{6 \delta}$, then inductively we obtain
$R_j \leq 2 R_1$. From Proposition \ref{prop:estimatesEvolutionSys}
it follows that for any $\lambda>0$, there exists $T_0 >0$ such that
$R_1<\lambda$. Thus we obtain a bound for $R_j$ uniformly in $j$,
provided $T_0$ is small enough. Using this uniform bound for $R_j$,
it follows that the sequences
$$
(t\mapsto t^{\gamma}u_j(t) )_{j\geq 1} \quad \mbox{and} \quad
(t\mapsto t^{\gamma+\frac{1}{2}}\nabla u_j(t) )_{j\geq 1}
$$
are uniformly bounded in $L^q_{\sigma}(\R^d)$
and $L^q(\R^{d})^{d\times d}$ respectively for $t\in[0,T_0]$ and all $j
\in \N$. Moreover, from (\ref{eq:behavior_ts_1}) and
(\ref{eq:behavior_ts_2}) we can conclude that the maps $t\mapsto
t^{\gamma}u_1(t)$ and $t\mapsto t^{\frac 1 2} \nabla u_1(t)$ are 
continuous at $t=0$. The continuity of $t\mapsto t^{\gamma}u_j(t)$
and $t\mapsto t^{\frac 1 2} \nabla u_j(t)$ for $j\geq 1$ now follows
by similar arguments as above.

We now derive estimates for the difference $u_{j+1}-u_j$.  First we
note that
$$
(u_{j}\cdot \nabla)u_{j} - (u_{j-1} \cdot \nabla) u_{j-1} =
(u_{j}\cdot \nabla)(u_{j}-u_{j-1}) + ((u_{j}-u_{j-1})\cdot
\nabla)u_{j-1}
$$
holds. Similarly as above we obtain
\begin{align*}
&\|u_{j+1}(t) - u_j(t)\|_{L^q(\R^d)} \\
&\quad \leq \int_0^t \|V(t,s)\mathbb P((u_{j}(s)\cdot \nabla)u_{j}(s) -(u_{j-1}(s)\cdot \nabla)u_{j-1}(s)) \|_{L^q(\R^d)}\d s \\
&\quad \leq C \int_0^t (t-s)^{-\frac d 2 \left(\frac 1 r - \frac{1}{q}\right)}\|(u_{j}(s)\cdot \nabla)u_{j}(s) -(u_{j-1}(s)\cdot \nabla)u_{j-1}(s)  \|_{L^r(\R^d)}\d s \\
& \quad\leq C \int_0^t (t-s)^{-\frac{d}{2p}}\Big(\| u_{j}(s)\|_{L^q(\R^d)} \| \nabla (u_{j}(s) -u_{j-1}(s))\|_{L^p(\R^d)}\\
& \quad\qquad \qquad+ \|u_{j}(s) - u_{j-1}(s)\|_{L^q(\R^d)} \|\nabla
u_{j-1}(s)\|_{L^p(\R^d)}\Big) \d s,
\end{align*}
and
\begin{align*}
&\|\nabla u_{j+1}(t) - \nabla u_j(t)\|_{L^p(\R^d)} \\
&\quad \leq \int_0^t \|\nabla V(t,s)\mathbb P((u_{j}(s)\cdot \nabla)u_{j}(s) -(u_{j-1}(s)\cdot \nabla)u_{j-1}(s)) \|_{L^p(\R^d)}\d s \\
& \quad\leq C \int_0^t (t-s)^{-\frac d 2 \left(\frac 1 r - \frac{1}{p}\right)-\frac 1 2}\|(u_{j}(s)\cdot \nabla)u_{j}(s) -(u_{j-1}(s)\cdot \nabla)u_{j-1}(s)  \|_{L^r(\R^d)}\d s \\
& \quad\leq C \int_0^t (t-s)^{-\frac{d}{2q}-\frac 1 2} \Big( \| u_{j}(s)\|_{L^q(\R^d)} \| \nabla (u_{j}(s) -u_{j-1}(s))\|_{L^p(\R^d)}\\
& \quad\qquad \qquad+ \|u_{j}(s) - u_{j-1}(s)\|_{L^q(\R^d)} \|\nabla
u_{j-1}(s)\|_{L^p(\R^d)} \Big) \d s.
\end{align*}
Thus we can conclude
\begin{equation}
L_{j} \leq C_3 (L_{j-1}'K_j  + L_{j-1} K_{j-1}') \leq 2C_3
R_1(L_{j-1}' + L_{j-1})
\end{equation}
and
\begin{equation}
L_{j}' \leq C_4 (L_{j-1}'K_j  + L_{j-1} K_{j-1}') \leq 2C_4
R_1(L_{j-1}' + L_{j-1}) ,
\end{equation}
for some positive constants $C_3,C_4$ independent of $j$, but
depending on $T$. These estimates show that if $R_1$ is sufficiently
small then the sequences $(t\mapsto t^{\gamma}u_j(t) )_{j\geq 1}$ and $(t\mapsto
t^{\gamma + \frac 1 2}\nabla u_j(t) )_{j\geq 1}$ are Cauchy
sequences in the spaces $C([0,T_0);L^q_{\sigma}(\R^d))$ and
$C([0,T_0);L^q(\R^d)^{d\times d})$, respectively. As it was previously mentioned, 
$R_1$ can be made sufficiently small if $\|u_0\|_{L^p(\R^d)}$ is small
enough or if we choose $T_0$ sufficiently small. As a consequence $t\mapsto
t^{\gamma}u_j(t)$ converges to some $t^{\gamma}u(t)\in
C([0,T_0),L^q_{\sigma}(\R^d))$ and $t\mapsto t^{\gamma + \frac 1
2}\nabla u_j(t) $ converges to some $ t^{\gamma + \frac 1 2}v(t)\in
C([0,T_0),L^q(\R^d)^{d\times d})$. It follows directly from the
construction that $v(t)=\nabla u(t)$ and that $u$ satisfies
(\ref{eq:Duhamel}) on $[0,T_0)$. The property
(\ref{eq:mild_solution_prop3}) follows from the construction and
Proposition \ref{prop:estimatesEvolutionSys}. Moreover, by
(\ref{eq:LpLqsmoothing}) and (\ref{eq:Hoelder}) we obtain
$$
\|u(t)\|_{L^p(\R^d)} \leq \|V(t,0)u_0\|_{L^p(\R^d)} + C \int_0^t (t-s)^{-\frac{d}{2q}}
s^{-\gamma - \frac 1 2}\d s,
$$
for some constant $C>0$ and thus $\sup_{0\leq t \leq T_0}\|u(t)\|_{L^p(\R^d)}
< \infty$ holds. The continuity at $0$ can be seen similarly, so $u\in
C([0,T_0);L^p_{\sigma}(\R^d))$.

It remains to prove the uniqueness of a mild solution
$u$ with the mentioned properties. To do this let $u,v$ be two mild solutions of
(\ref{eq:NS_abstract}) satisfying (\ref{eq:mild_solution_prop1}),
(\ref{eq:mild_solution_prop2}) and (\ref{eq:mild_solution_prop3}).
Moreover, let $0<\tilde T \leq T_0$ and define the constant
$K$ as
$$
K:=K(\tilde T) := \max\Big\{ \sup_{0< t \leq \tilde T}t^{\gamma}\|u(t)\|_{L^q(\R^d)}, \;\sup_{0< t \leq \tilde T}t^{\frac 1 2}\|\nabla
v(t)\|_{L^p(\R^d)}\Big\}.
$$
Since $u$ and $v$ both solve the integral equation
(\ref{eq:Duhamel}), we obtain similarly as above \small
\begin{align*}
&\|u(t)-v(t)\|_{L^q(\R^d)} 
 \leq K C \Big(\int_0^t (t-s)^{-\frac{d}{2p}} s^{-\gamma-\frac 1
2} \d s\Big)\cdot \\
& \qquad \quad \sup_{0< \tau \leq \tilde T}\Big(\tau^{\gamma}
\|u(\tau) - v(\tau)\|_{L^q(\R^d)} + \tau^{\frac 1 2}\|\nabla
(u(\tau)-v(\tau))\|_{L^p(\R^d)} \Big),
\end{align*}\normalsize
and\small
\begin{align*}
&\|\nabla u(t) - \nabla v(t)\|_{L^p(\R^d)}  
 \leq K C \Big(\int_0^t (t-s)^{-\frac{d}{2q}-\frac 1 2} s^{-\gamma-\frac 1
2} \d s\Big)\cdot \\
& \qquad \quad \sup_{0< \tau \leq \tilde T}\Big(\tau^{\gamma}
\|u(\tau) - v(\tau)\|_{L^q(\R^d)} + \tau^{\frac 1 2}\|\nabla
(u(\tau)-v(\tau))\|_{L^p(\R^d)} \Big),
\end{align*}\normalsize
for $0< t \leq \tilde T$. Thus, for $0<t\leq \tilde T$ we have\small
\begin{align}
&t^{\gamma}\|u(t)-v(t)\|_{L^q(\R^d)} + t^{\frac 1 2}
\|\nabla\left(u(t)-v(t)\right)\|_{L^p(\R^d)}\label{eq:uniqueness_1}\\
&\qquad \leq 2 K C \tilde T^{1-\frac{d}{2p}-\frac 1 2}
\sup_{0< \tau \leq \tilde T}\left(\tau^{\gamma} \|u(\tau) -
v(\tau)\|_{L^q(\R^d)} + \tau^{\frac 1 2}\|\nabla (u(\tau)-v(\tau))\|_{L^p(\R^d)}
\right).\notag
\end{align}\normalsize
In the case $p>d$ we can choose $\tilde T$ small, so that $2 K C
\tilde T^{1-\frac{d}{2p}-\frac 1 2}  < 1$. This implies $u=v$ on
$[0,\tilde T)$. Since $u,v \in C([\eps, T_0);L^q_{\sigma}(\R^d))$
for every $\eps > 0$, the above argument with initial data
$u(\eps)=v(\eps)$ yields that the set $\{t \in (0,T_0): u(t)=v(t)\}$
is open. The continuity of $u,v$ and the connectedness of $(0,T_0)$ imply that $u=v$ on $[0,T_0)$.

Now, it remains to prove the uniqueness in the case $p=d$. Instead of
(\ref{eq:uniqueness_1}) we consider
\begin{align*}
&t^{\gamma}\|u(t)-v(t)\|_{L^q(\R^d)} + t^{\frac 1 2}
\|\nabla\left(u(t)-v(t)\right)\|_{L^p(\R^d)} \\
&\qquad\qquad \leq 2 K C \,\sup_{0< \tau \leq \tilde T}\Big(\|\nabla
(u(\tau)-v(\tau))\|_{L^p(\R^d)} + \|u(\tau) - v(\tau)\|_{L^q(\R^d)} \Big)
\end{align*}
for $0< t \leq \tilde T$. By (\ref{eq:mild_solution_prop3}) the
constant $K:=K(\tilde T)$ tends to zero as $\tilde T \rightarrow 0$.
Thus, we can choose $\tilde T$ small, so that $2 K C < 1$. This
shows $u=v$ on $[0,\tilde T)$. Since $u,v \in C([\tilde T /2,
T_0);L^q_{\sigma}(\R^d))$ for $q>d$ with $u(\tilde T /2) = v(\tilde
T /2)$, the uniqueness in the case $p>d$ implies $u=v$ on $[\tilde T
/2, T_0)$. The proof is hence complete.
\end{proof}
If we assume in addition that $M(t)$ is skew-symmetric for all $t>0$, then we can even expect to obtain a global solution, provided
that $u_0 \in L^d_{\sigma}(\R^d)$ and that $\|u_0\|_{L^d(\R^d)}$ is
sufficiently small.
\begin{prop}
Let $d\geq 2$ and $u_0\in L^d_{\sigma}(\R^d)$. Moreover assume that
$M(t)$ is skew-symmetric for all $t>0$. Then there exists
$\lambda>0$, such that if $\|u_0\|_{L^d(\R^d)} < \lambda$, then the mild
solution $u\in C([0,T_0);L^d_{\sigma}(\R^d))$ obtained in
Proposition \ref{prop:Kato} is global, i.e. we may take $T_0 =
+\infty$.
\end{prop}
For the proof one can use the estimates (\ref{eq:LpLqsmoothing}) and
(\ref{eq:GradientEstimate}) and the same argumentation as above.


\begin{thebibliography}{MPRS02}

\bibitem[Acq84]{Acquistapace:1984}
Paolo Acquistapace.
\newblock Some existence and regularity results for abstract nonautonomous
  parabolic equations.
\newblock {\em J. Math. Anal. Appl.}, 99(1):9--64, 1984.

\bibitem[AT86]{Acquistapace/Terreni:1986}
Paolo Acquistapace and Brunello Terreni.
\newblock On fundamental solutions for abstract parabolic equations.
\newblock In {\em Differential equations in {B}anach spaces ({B}ologna, 1985)},
  volume 1223 of {\em Lecture Notes in Math.}, pages 1--11. Springer, Berlin,
  1986.

\bibitem[AT87]{Acquistapace/Terreni:1987}
Paolo Acquistapace and Brunello Terreni.
\newblock A unified approach to abstract linear nonautonomous parabolic
  equations.
\newblock {\em Rend. Sem. Mat. Univ. Padova}, 78:47--107, 1987.

\bibitem[Bor92]{Borcher:1992}
Wolfgang Borchers.
\newblock {\em Zur Stabilit\"at und Faktorisierungsmethode f\"ur die
  Navier-Stokes-Gleichungen inkompressibler viskoser Fl\"ussigkeiten}.
\newblock Habiliationschrift, Universit\"at Paderborn, 1992.

\bibitem[DPL07]{DaPrato/Lunardi:2007}
Giuseppe Da~Prato and Alessandra Lunardi.
\newblock Ornstein-{U}hlenbeck operators with time periodic coefficients.
\newblock {\em J. Evol. Equ.}, 7(4):587--614, 2007.

\bibitem[Far06]{Farwig:2005}
Reinhard Farwig.
\newblock An {$L\sp q$}-analysis of viscous fluid flow past a rotating
  obstacle.
\newblock {\em Tohoku Math. J. (2)}, 58(1):129--147, 2006.

\bibitem[Gal94]{Galdi:1994}
Giovanni~P. Galdi.
\newblock {\em An introduction to the mathematical theory of the
  {N}avier-{S}tokes equations. {V}ol. {I}}, volume~38 of {\em Springer Tracts
  in Natural Philosophy}.
\newblock Springer-Verlag, New York, 1994.
\newblock Linearized steady problems.

\bibitem[GHH06]{Geissert/Heck/Hieber:2006}
Matthias Geissert, Horst Heck, and Matthias Hieber.
\newblock {$L\sp p$}-theory of the {N}avier-{S}tokes flow in the exterior of a
  moving or rotating obstacle.
\newblock {\em J. Reine Angew. Math.}, 596:45--62, 2006.

\bibitem[Gig86]{Giga:1986}
Yoshikazu Giga.
\newblock Solutions for semilinear parabolic equations in {$L\sp p$} and
  regularity of weak solutions of the {N}avier-{S}tokes system.
\newblock {\em J. Differential Equations}, 62(2):186--212, 1986.

\bibitem[GL08]{Geissert/Lunardi:2007}
Matthias Geissert and Alessandra Lunardi.
\newblock Invariant measures and maximal {$L\sp 2$} regularity for
  nonautonomous {O}rnstein-{U}hlenbeck equations.
\newblock {\em J. Lond. Math. Soc. (2)}, 77(3):719--740, 2008.

\bibitem[His99]{Hishida:1999}
Toshiaki Hishida.
\newblock An existence theorem for the {N}avier-{S}tokes flow in the exterior
  of a rotating obstacle.
\newblock {\em Arch. Ration. Mech. Anal.}, 150(4):307--348, 1999.

\bibitem[His01]{Hishida:2001}
Toshiaki Hishida.
\newblock On the {N}avier-{S}tokes flow around a rigid body with a prescribed
  rotation.
\newblock In {\em Proceedings of the {T}hird {W}orld {C}ongress of {N}onlinear
  {A}nalysts, {P}art 6 ({C}atania, 2000)}, volume~47, pages 4217--4231, 2001.

\bibitem[HS05]{Hieber/Sawada:2005}
Matthias Hieber and Okihiro Sawada.
\newblock The {N}avier-{S}tokes equations in {$\Bbb R\sp n$} with linearly
  growing initial data.
\newblock {\em Arch. Ration. Mech. Anal.}, 175(2):269--285, 2005.

\bibitem[HS09]{Hishida/Shibata:2009}
Toshiaki Hishida and Yoshihiro Shibata.
\newblock ${L}_ p$-${L}_q$ {E}stimate of the {S}tokes {O}perator and
  {N}avier{}{S}tokes {F}lows in the {E}xterior of a {R}otating {O}bstacle.
\newblock {\em Arch. Ration. Mech. Anal.}, to appear, 2009.

\bibitem[Kat84]{Kato:1984}
Tosio Kato.
\newblock Strong {$L\sp{p}$}-solutions of the {N}avier-{S}tokes equation in
  {${\bf R}\sp{m}$}, with applications to weak solutions.
\newblock {\em Math. Z.}, 187(4):471--480, 1984.

\bibitem[Met01]{Metafune:2001}
Giorgio Metafune.
\newblock {$L\sp p$}-spectrum of {O}rnstein-{U}hlenbeck operators.
\newblock {\em Ann. Scuola Norm. Sup. Pisa Cl. Sci. (4)}, 30(1):97--124, 2001.

\bibitem[MPRS02]{Metafune/etal:2002}
Giorgio Metafune, Jan Pr{\"u}ss, Abdelaziz Rhandi, and Roland Schnaubelt.
\newblock The domain of the {O}rnstein-{U}hlenbeck operator on an {$L\sp
  p$}-space with invariant measure.
\newblock {\em Ann. Sc. Norm. Super. Pisa Cl. Sci. (5)}, 1(2):471--485, 2002.

\bibitem[Paz83]{Pazy:1983}
A.~Pazy.
\newblock {\em Semigroups of linear operators and applications to partial
  differential equations}, volume~44 of {\em Applied Mathematical Sciences}.
\newblock Springer-Verlag, New York, 1983.

\bibitem[Shi08]{Shibata:2008}
Yoshihiro Shibata.
\newblock On the {O}seen semigroup with rotating effect.
\newblock In {\em Functional analysis and evolution equations}, pages 595--611.
  Birkh\"auser, Basel, 2008.

\bibitem[Tan59]{Tanabe:1959}
Hiroki Tanabe.
\newblock A class of the equations of evolution in a {B}anach space.
\newblock {\em Osaka Math. J.}, 11:121--145, 1959.

\bibitem[Tan60a]{Tanabe:1960a}
Hiroki Tanabe.
\newblock On the equations of evolution in a {B}anach space.
\newblock {\em Osaka Math. J.}, 12:363--376, 1960.

\bibitem[Tan60b]{Tanabe:1960b}
Hiroki Tanabe.
\newblock Remarks on the equations of evolution in a {B}anach space.
\newblock {\em Osaka Math. J.}, 12:145--166, 1960.

\bibitem[Tan97]{Tanabe:1997}
Hiroki Tanabe.
\newblock {\em Functional analytic methods for partial differential equations},
  volume 204 of {\em Monographs and Textbooks in Pure and Applied Mathematics}.
\newblock Marcel Dekker Inc., New York, 1997.

\bibitem[TG06]{Thomann/Guenther:2006}
Enrique~A. Thomann and Ronald~B. Guenther.
\newblock The fundamental solution of the linearized {N}avier-{S}tokes
  equations for spinning bodies in three spatial dimensions---time dependent
  case.
\newblock {\em J. Math. Fluid Mech.}, 8(1):77--98, 2006.

\end{thebibliography}
\end{document}